\documentclass{amsart}

\usepackage[section]{placeins}

\usepackage{amsmath}             
\usepackage{amscd}
\usepackage{amssymb}             
\usepackage{amsfonts}            
\usepackage{latexsym}            
\usepackage{amsthm}              
\usepackage{graphicx}

\usepackage{enumerate}
\usepackage{mathrsfs} 
\usepackage{thmtools}
\usepackage{thm-restate}

\usepackage{hyperref}

\newcommand{\C}{\mathbb{C}}      
   
\newcommand{\F}{\mathbb{F}}      
          
\newcommand{\Sp}{\operatorname{Sp}}
\newcommand{\SL}{\operatorname{SL}}
\newcommand{\GL}{\operatorname{GL}}
\newcommand{\GU}{\operatorname{GU}}

\newcommand{\legendre}[2]{\genfrac{(}{)}{}{}{#1}{#2}}

\newtheorem{lause}{Theorem}[section]

\newtheorem{lemma}[lause]{Lemma}

\newtheorem*{lause*}{Theorem}

\theoremstyle{definition}

\theoremstyle{remark}
\newtheorem{remark}[lause]{Remark}
 
\newtheorem*{mot*}{Motivation}
\newtheorem*{acknow*}{Acknowledgements}

\numberwithin{equation}{section}

\usepackage{hyphenat}
\allowdisplaybreaks

\begin{document}

\title[Matrix generators for Weil representations]{Matrix generators for Weil representations}  

\author{Mikko Korhonen}
\thanks{\emph{Present address:} Mathematics Research Centre, Tampere University, 33720 Tampere, Finland}

\address{M. Korhonen, Shenzhen International Center for Mathematics, Southern University of Science and Technology, Shenzhen 518055, Guangdong, P. R. China}
\email{korhonen\_mikko@hotmail.com {\text{\rm(Korhonen)}}}
\thanks{}

\date{\today}

\begin{abstract}
Let $r$ be an odd prime and $\mathbb{F}$ a field containing a primitive $r$th root of unity. Then for all $\ell \geq 1$, there is a faithful representation $f: \operatorname{Sp}_{2\ell}(r) \rightarrow \operatorname{GL}_{r^\ell}(\mathbb{F})$ called the \emph{Weil representation}. We provide explicit matrices generating $\operatorname{Sp}_{2\ell}(r)$ in $\operatorname{GL}_{r^\ell}(\mathbb{F})$, which we have implemented in Magma. We also describe such generators for the \emph{irreducible Weil representations} of $\operatorname{Sp}_{2\ell}(r)$, which are of degree $(r^{\ell} \pm 1)/2$ and arise as irreducible constituents of the Weil representations.
\end{abstract}

\maketitle

\section{Introduction} 

Let $r$ be an odd prime, and let $\F$ be a field containing a primitive $r$th root of unity. Then for all $\ell \geq 1$, there is an embedding of $\Sp_{2\ell}(r)$ into $\GL_{r^{\ell}}(\F)$, arising from the \emph{Weil representation} of $\Sp_{2\ell}(r)$. The purpose of this note is to provide explicit matrices that generate this copy of $\Sp_{2\ell}(r)$ in $\GL_{r^{\ell}}(\F)$. Originating from the work of Weil in \cite{Weil1964}, the Weil representation has been extensively studied in the literature, and there are several different proofs of its existence and construction \cite{BoltRoomWall1961}, \cite{Ward1972}, \cite{Howe1973}, \cite{Isaacs1973}, \cite{Gerardin1977}, \cite{GlasbyHowlett1992}, \cite{Szechtman1998}. The construction described in this note gives yet another proof of the existence of the Weil representation. However, our main motivation here will be a computational implementation. We have implemented the generators in Magma \cite{MAGMA}, and the code is available at \cite{KorhonenGithub}.

To describe the generators, we will first have to setup some basic facts about representations of extraspecial groups, we refer to \cite[Chapter A, Section 20 and Chapter B, Section 9]{DoerkHawkes} for more details. Let $R$ be an extraspecial $r$-group of exponent $r$, so $|R| = r^{1+2\ell}$ for some $\ell \geq 1$. We note that $R$ is often denoted by $R = r_{+}^{1+2\ell}$ in the literature. Then every irreducible $\F[R]$-module is absolutely irreducible, and every irreducible $\F[R]$-module of dimension $> 1$ is faithful of dimension $r^{\ell}$ \cite[Chapter B, Theorem 9.16]{DoerkHawkes}. Moreover faithful irreducible $\F[R]$-modules are quasi-equivalent \cite[A, 20.8 (d) and B, 9.16(ii)]{DoerkHawkes}, which implies that they define a unique conjugacy class of subgroups of $\GL_{r^{\ell}}(\F)$.

Suppose then that $R < \GL(W)$, where $W$ is an absolutely irreducible $\F[R]$-module of dimension $r^{\ell}$. Let $\theta$ be a primitive $r$th root of unity in $\F$. Note that since the center $Z(R)$ is cyclic of order $r$, we have $Z(R) = \langle \theta I_{W} \rangle$. The basic observation needed to understand the structure of $N_{\GL(W)}(R)$ is that the quotient $R/Z(R)$ can be considered as a vector space of dimension $2\ell$ over the finite field $\mathbb{F}_r$ of order $r$. The group commutator then induces a non-degenerate alternating bilinear form $b: R/Z(R) \times R/Z(R) \rightarrow \mathbb{F}_r$ on $R/Z(R)$, with $$[x,y] = \theta^{b(xZ(R), yZ(R))}$$ for all $x,y \in R$. 

It is clear that elements of $\GL(W)$ normalizing $R$ will leave invariant the bilinear form $b$ defined on $R/Z(R)$, so we have a homomorphism \begin{equation}\label{eq:pimap}\pi: N_{\GL(W)}(R) \rightarrow \Sp_{2\ell}(r),\end{equation} where $\pi(g)$ is defined as the action of $g$ on $R/Z(R)$, for all $g \in N_{\GL(W)}(R)$.

It is then well known that the following hold, see for example \cite[Lemma 3.1.1, Theorem 3.2.2 (ii)]{KorhonenBook}. \begin{align} \label{eq:kerpi}&\operatorname{Ker} \pi = RZ, \text{ where } Z \cong \F^\times \text{ is the subgroup of scalar matrices;} \\ \label{eq:kerpiquotient}&N_{\GL(W)}(R)/\operatorname{Ker} \pi \cong \Sp_{2\ell}(r). \end{align} It turns out that $N_{\GL(W)}(R)$ is a split extension of $\operatorname{Ker} \pi$, meaning that the normalizer contains a subgroup $G \cong \Sp_{2\ell}(r)$ with $G \cap \operatorname{Ker} \pi = 1$. This subgroup $G$ then defines the Weil representation of $\Sp_{2\ell}(r)$.

To describe the generators for $G \cong \Sp_{2\ell}(r)$, we first have to describe the embedding $R < \GL(W)$. We will also need a well-known set of generators for $N_{\GL(W)}(R)$, which go back to the work of Jordan (see Remark \ref{remark:history}). Let $V$ be an $r$-dimensional $\F$-vector space with basis $v_0$, $v_1$, $\ldots$, $v_{r-1}$. We define linear maps $A,B,C,E: V \rightarrow V$ by \begin{equation} \label{eq:maps1} \begin{aligned} Av_{\xi} &= \theta^{\xi} v_{\xi} \\ Bv_{\xi} &= v_{\xi+1} \\ C v_{\xi} &= \sum_{0 \leq i < r} \theta^{i \xi} v_i \\ E v_{\xi} &= \theta^{\frac{\xi(\xi-1)}{2}} v_{\xi} \end{aligned} \end{equation} for all $0 \leq \xi < r$, where the indices are interpreted modulo $r$. It is clear that $A,B,E$ are invertible, and $C$ is also invertible by the Vandermonde determinant.

Then $R_0 = \langle A,B \rangle \cong r_{+}^{1+2}$ is absolutely irreducible, and is normalized by $C$ and $E$. Let $W = V \otimes \cdots \otimes V$, where $V$ appears $\ell$ times in the tensor product. Since $R_0$ is absolutely irreducible, the tensor product $R = R_0 \otimes \cdots \otimes R_0$ ($\ell$ times) is an absolutely irreducible subgroup of $\GL(W)$, see for example \cite[Lemma 4.4.3 (vi)]{KleidmanLiebeck}. Here $R$ is a central product of $\ell$ copies of $r_{+}^{1+2}$, which implies that $R \cong r_{+}^{1+2\ell}$.

As a basis for $W$, we take the vectors $v_{\xi_1, \ldots, \xi_{\ell}} := v_{\xi_1} \otimes \cdots \otimes v_{\xi_{\ell}}$ for $0 \leq \xi_1, \ldots, \xi_{\ell} < r$. We define linear maps $A_t, B_t, C_t, D_{st}, E_t: W \rightarrow W$ for all $1 \leq t \leq \ell$ and $1 \leq s < t \leq \ell$ by \begin{equation}\label{eq:ABCDE}\begin{aligned} A_t v_{\xi_1, \ldots, \xi_{\ell}} &= \theta^{\xi_t} v_{\xi_1, \ldots, \xi_{\ell}} \\ B_t v_{\xi_1, \ldots, \xi_{\ell}} &= v_{\xi_1, \ldots, \xi_{t-1}, \xi_{t}+1, \xi_{t+1}, \ldots, \xi_{\ell}} \\ C_t v_{\xi_1, \ldots, \xi_{\ell}} &= \sum_{0 \leq i < r} \theta^{i \xi_t} v_{\xi_1, \ldots, \xi_{t-1}, i, \xi_{t+1}, \ldots, \xi_{\ell}} \\ D_{st} v_{\xi_1, \ldots, \xi_{\ell}} &= \theta^{\xi_s \xi_t} v_{\xi_1, \ldots, \xi_{\ell}} \\ E_t v_{\xi_1, \ldots, \xi_{\ell}} &= \theta^{\frac{\xi_t(\xi_t-1)}{2}} v_{\xi_1, \ldots, \xi_{\ell}} \end{aligned}\end{equation} for all $0 \leq \xi_1, \ldots, \xi_{\ell} < r$, where as before the indices are interpreted modulo $r$. Note that $A_t = I_{r^{t-1}} \otimes A \otimes I_{r^{\ell-t}}$, $B_t = I_{r^{t-1}} \otimes B \otimes I_{r^{\ell-t}}$, $C_t = I_{r^{t-1}} \otimes C \otimes I_{r^{\ell-t}}$, $E_t = I_{r^{t-1}} \otimes E \otimes I_{r^{\ell-t}}$ for all $1 \leq t \leq \ell$.

The following theorem is to a large extent due to Jordan, see for example \cite[Theorem 3.2.2]{KorhonenBook} for a proof.

\begin{lause}\label{thm:jordanthm}
Let $Z$ be the group of scalar matrices in $\GL(W)$, and let $R$ be the subgroup of $\GL(W)$ generated by the linear maps $A_t$ and $B_t$. Then the following statements hold:
	\begin{enumerate}[\normalfont (i)]
		\item $R < \GL(W)$ is an absolutely irreducible extraspecial group of exponent $r$, and order $r^{1+2\ell}$.
		\item $N_{\GL(W)}(R)$ is generated by $Z$ together with the linear maps $A_t$, $B_t$, $C_t$, $D_{st}$, $E_t$;
		\item $N_{\GL(W)}(R) / RZ \cong \Sp_{2\ell}(r)$. 
	\end{enumerate}
\end{lause}

\begin{remark}\label{remark:history}
The generators $A_t$, $B_t$, $C_t$, $D_{st}$, $E_t$ defined in~\eqref{eq:ABCDE} are originally due to Jordan, see \cite[(15), p. 437, and (25)--(27), p. 441]{jordantraite} and \cite[\S X]{jordan1908recherches}. Later similar generators have appeared in various references, for example in \cite[Section 9]{HoltRoneyDougal2005}, \cite[Section 4.6]{KleidmanLiebeck}, and \cite[Theorem 21.4]{SuprunenkoBook}.
\end{remark}

We will now describe generators for $G \cong \Sp_{2\ell}(r)$. First on the representation $V$ of $R_0 = r_{+}^{1+2}$ we started with, we define $U: V \rightarrow V$ by $U = A^{(r+1)/2} E$, so $$U v_{\xi} = \theta^{\frac{\xi(\xi+r)}{2}} v_{\xi}$$ for all $0 \leq \xi < r$. Then for $1 \leq t \leq \ell$ we define similarly $U_t: W \rightarrow W$ by $U_t = A_t^{(r+1)/2} E_t$, so \begin{equation}\label{eq:U}U_t v_{\xi_1, \ldots, \xi_{\ell}} = \theta^{\frac{\xi_t(\xi_t+r)}{2}} v_{\xi_1, \ldots, \xi_{\ell}}\end{equation} for all $0 \leq \xi_1, \ldots, \xi_{\ell} < r$. Note that then $U_t = I_{r^{t-1}} \otimes U \otimes I_{r^{\ell-t}}$ for all $1 \leq t \leq \ell$.

Next we define a scalar multiplier \begin{equation}\label{eq:lambda}\lambda = (-r)^{(r-1)/2} \det(C)^{-1}\end{equation} that we will use to change each $C_t$ to a linear map of determinant one. Our main result is the following.

\begin{lause}\label{thm:mainthm}
Let $G$ be the subgroup of $\GL(W)$ generated by the linear maps $\lambda C_t$, $D_{st}$, and $U_t$. Then $G \cong \Sp_{2\ell}(r)$, with an isomorphism given by restriction of the map $\pi$ to $G$.
\end{lause}

Thus by Theorem \ref{thm:mainthm}, with~\eqref{eq:ABCDE} and~\eqref{eq:U} we can write down generating matrices for $G$. We can then use this to get the action of every $g \in \Sp_{2\ell}(r)$ in the Weil representation --- see Remark \ref{remark:generatorswrite}.

One can also describe the submodules of $G$ on $W$ explicitly, thus giving the action of the generators of $G$ on the \emph{irreducible Weil representations}, which are irreducible representations of $\Sp_{2\ell}(r)$ of degree $(r^{\ell} \pm 1)/2$. Specifically, if $\operatorname{char} \F \neq 2$, then $W$ is a direct sum of two irreducible $\F[G]$-submodules of degrees $(r^{\ell} + 1)/2$ and $(r^{\ell}-1)/2$. And if $\operatorname{char} \F = 2$, then $W$ is uniserial with composition factors of degrees $(r^{\ell}-1)/2$, $1$, and $(r^{\ell}-1)/2$. We discuss the $\F[G]$-submodules of $W$ in Section \ref{section:submodules}.

About the proof of Theorem \ref{thm:mainthm}, we note that since $U_t = A_t^{(r+1)/2} E_t$, it is clear from Theorem \ref{thm:jordanthm} that the image of $G$ under the map $\pi$ is equal to $\Sp_{2\ell}(r)$. Thus to prove Theorem \ref{thm:mainthm}, it suffices to check that $G \cong \Sp_{2\ell}(r)$. This will be done in Section \ref{section:mainthm}, where our strategy for the proof of Theorem \ref{thm:jordanthm} is as follows. For $(r,\ell) = (3,1)$ we will verify that the generators of $G$ satisfy relations defining $\Sp_2(3) = \SL_2(3)$. In the case $(r,\ell) \neq (3,1)$, by calculations with the generators we will check that $G$ is a perfect central extension of $\Sp_{2\ell}(r)$. Since $\Sp_{2\ell}(r)$ has trivial Schur multiplier \cite[Theorem 1.1]{Steinberg1981}, we can then conclude that $G \cong \Sp_{2\ell}(r)$.

\begin{remark}
Here the determinant involved in the definition~\eqref{eq:lambda} of $\lambda$ is a Vandermonde determinant, so $$\det(C) = \prod_{0 \leq i < j < r} (\theta^j - \theta^i).$$ Over $\F = \C$, the precise evaluation of this determinant goes back to Schur, who used it to determine the value of the quadratic Gauss sum. Schur proved in \cite[p. 149]{Schur1921} that for $\theta = \exp(2\pi i / r)$ we have $\det(C) = r^{r/2} i^{r(r-1)/2}$, see for example \cite[pp. 211--212]{Landau1958}. (There is a version of this formula that works when $r$ is replaced by an arbitrary positive integer, see for example \cite[Section 6.3, p. 109]{Rose1994})

Over an arbitrary field $\F$ containing a primitive $r$th root of unity, we have $\det(C)^2 = r^r (-1)^{r(r-1)/2}$, so $\det(C)$ is still determined up to a sign as the square root of $r^r (-1)^{r(r-1)/2}$. We also have the formula $$\det(C) = \legendre{2}{r} r^{(r-1)/2} \sum_{0 \leq i < r} \theta^{i^2},$$ where $\legendre{a}{r}$ is the Legendre symbol --- see Remark \ref{remark:detformula}. 

\end{remark}

\begin{remark}\label{remark:generatorswrite}
Let $e_i$ and $f_i$ be the images in $R/Z(R)$ for $A_i$ and $B_i$, respectively. We have $[A_i,B_j] = \theta^{\delta_{i,j}}$ and $[A_i,A_j] = [B_i,B_j] = 1$ for all $1 \leq i,j \leq \ell$, so with respect to the alternating bilinear form defined on $R/Z(R)$ the vectors $e_1$, $f_1$, $\ldots$, $e_{\ell}$, $f_{\ell}$ form a basis of hyperbolic pairs. With respect to this basis, the action of $C_t$, $D_{st}$, and $E_t$ is given as follows: \begin{align*} \pi(C_t)&: \begin{cases} e_t \mapsto -f_t \\ f_t \mapsto e_t \end{cases} & \pi(D_{st}): \begin{cases} f_t \mapsto f_t + e_s \\ f_s \mapsto f_s + e_t \end{cases} \\ \pi(E_t)&: f_t \mapsto f_t + e_t \end{align*} with all the other basis vectors fixed (see for example \cite[(3.1), p. 74]{KorhonenBook}). Note that $\pi(E_t) = \pi(U_t)$ for all $1 \leq t \leq \ell$.

By Theorem \ref{thm:jordanthm}, we know that the elements $\pi(C_t)$, $\pi(D_{st})$, $\pi(E_t)$ generate $\Sp_{2\ell}(r)$. (This was first proven by Jordan in \cite[Th\'eor\`eme 221, p. 174]{jordantraite}.) Moreover, there is a fairly straightforward algorithm for writing down a given element of $\Sp_{2\ell}(r)$ as a product of the generators $\pi(C_t)$, $\pi(D_{st})$, and $\pi(E_t)$. We omit the details, but such an algorithm can be extracted for example from the proof of \cite[Theorem 3.2.2]{KorhonenBook}.

It follows then that for a given element $g \in \Sp_{2\ell}(r)$ written as a word in the generators $\pi(C_t)$, $\pi(D_{st})$, $\pi(E_t) = \pi(U_t)$, evaluating the corresponding word in the generators $\lambda C_t$, $D_{st}$, $U_t$ gives a matrix $g' \in G$ which is the action of $g$ in the Weil representation of $\Sp_{2\ell}(r)$. Furthermore, using the description of the $\F[G]$-submodules of $W$ discussed in Section \ref{section:submodules}, one can similarly calculate the action of $g$ on the irreducible Weil representations of degrees $(r^{\ell} \pm 1)/2$.
\end{remark}

\begin{remark}\label{remark:nonprimefield}
Let $s \geq 1$ be an integer. More generally, one defines Weil representations of $\Sp_{2\ell}(r^s)$ as follows. We can consider $\Sp_{2\ell}(r^s)$ as a field extension subgroup $\Sp_{2\ell}(r^s) \leq \Sp_{2\ell s}(r)$, see for example \cite[Section 4.3]{KleidmanLiebeck} for the definition and construction. Then the Weil representations of $\Sp_{2\ell}(r^s)$ can be defined as the restriction of the Weil representations of $\Sp_{2\ell s}(r)$, and it is known that $\Sp_{2\ell s}(r)$ and $\Sp_{2\ell}(r^s)$ have the same submodules on $W$. Moreover, generators for $\Sp_{2\ell}(r^s)$ as a subgroup of $\Sp_{2\ell s}(r)$ are well known, and have been implemented in Magma \cite[Proposition 6.4]{HoltRoneyDougal2005}. Using this embedding $\Sp_{2\ell}(r^s) \leq \Sp_{2\ell s}(r)$, we have also included the construction of Weil representations of $\Sp_{2\ell}(r^s)$ in the Magma code available at \cite{KorhonenGithub}.

Similarly, with the embeddings $\GL_{\ell}(r^s) \leq \Sp_{2\ell}(r^s)$ and $\GU_{\ell}(r^s) \leq \Sp_{2\ell}(r^s)$, we can define Weil representations of finite general linear groups and unitary groups. Explicit matrix generators for these subgroups of $\Sp_{2\ell}(r^s)$ are also known and implemented in Magma, see \cite[Proposition 5.3, Proposition 6.5]{HoltRoneyDougal2005}. Thus one can similarly write down matrix generators for the Weil representations of $\GL_{\ell}(r^s)$ and $\GU_{\ell}(r^s)$ over finite fields.
\end{remark}

\begin{acknow*}
The author would like to thank Jay Taylor and an anonymous referee for their useful comments and suggestions. The work in this paper was done while the author was working at the Shenzhen International Center for Mathematics at the Southern University of Science and Technology.
\end{acknow*}

\section{Some preliminary calculations}

In this section, we provide some basic properties of the generators of the group $G$ from Theorem \ref{thm:mainthm}. All except the last lemma of this section are about the case $\ell = 1$, in which case $G = \langle \lambda C, U \rangle$.

\begin{lemma}\label{lemma:C2}
For all $0 \leq \xi < r$, we have $C^2 v_{\xi} = r v_{-\xi}$. 
\end{lemma}

\begin{proof}
For $0 \leq i,j < r$, the $(i,j)$ entry of the matrix of $C^2$ (with respect to the basis $v_0$, $v_1$, $\ldots$, $v_{r-1}$) is equal to \begin{equation}\label{eq:c2sum}\sum_{0 \leq \xi < r} \theta^{i \xi + \xi j} = \sum_{0 \leq \xi < r} \theta^{\xi(i+j)}.\end{equation} Now if $i+j \not\equiv 0 \mod{r}$, then~\eqref{eq:c2sum} is the sum of all $r$th roots of unity, and thus equals zero. If $i+j \equiv 0 \mod{r}$, then~\eqref{eq:c2sum} evaluates to $\sum_{0 \leq \xi < r} 1 = r$. From this the claim of the lemma follows.
\end{proof}

\begin{lemma}\label{lemma:detC2}
We have $\det(C^2) = (-1)^{(r-1)/2} r^r$, and $(\lambda C)^2 v_{\xi} = (-1)^{(r-1)/2} v_{-\xi}$ for all $0 \leq \xi < r$.
\end{lemma}

\begin{proof}
By Lemma \ref{lemma:C2}, we have $C^2 = r P$, where $P$ is a permutation matrix that fixes the basis vector $v_0$ and swaps $v_i$ with $v_{r-i}$ for $0 < i < r$. The permutation corresponding to $P$ is a product of $(r-1)/2$ transpositions, so we conclude that $\det(C^2) = (-1)^{(r-1)/2} r^r$.

For the second claim, note that $\lambda^2 = (-r)^{r-1} \det(C)^{-2} = r^{-1} (-1)^{(r-1)/2}$ by the first claim. Thus $(\lambda C)^2 v_{\xi} = (-1)^{(r-1)/2} v_{-\xi}$ follows from Lemma \ref{lemma:C2}.
\end{proof}

\begin{lemma}\label{lemma:detlemma}
We have $\det(\lambda C) = 1$. Moreover $\det(U) = 1$ if $r > 3$, and $\det(U) = \theta$ if $r = 3$.
\end{lemma}

\begin{proof}
First note that $$\det(\lambda C) = \lambda^r \det(C) = (-r)^{r(r-1)/2} \det(C)^{-r} \det(C) = \left(-r^r \det(C)^{-2}\right)^{(r-1)/2}.$$ Since $\det(C)^2 = (-1)^{(r-1)/2} r^r$ (Lemma \ref{lemma:detC2}), we get that $\det(\lambda C) = 1$. 

For the determinant of $U$, we have $$\det(U) = \prod_{0 \leq \xi < r} \theta^{\frac{\xi(\xi+r)}{2}} = \theta^{\sum_{0 \leq \xi < r} \xi(\xi+r)/2}.$$ Now $\sum_{0 \leq \xi < r} \xi^2 = (r-1)r(2r-1)/6$ is divisible by $r$ if $r > 3$, so it follows that $\det(U) = 1$. For $r = 3$ we have $U$ equal to a diagonal matrix $\operatorname{diag}(1, \theta^2, \theta^2)$, so $\det(U) = \theta$.
\end{proof}

\begin{lemma}\label{lemma:tracesquare}
We have $\operatorname{Tr}(U)^2 = (-1)^{(r-1)/2} r$.
\end{lemma}

\begin{proof}
First note that $\operatorname{Tr}(U)^2$ equals \begin{equation}\label{eq:doublesum}\left(\sum_{0 \leq i < r } \theta^{\frac{i(i+r)}{2}} \right)^2 = \sum_{0 \leq i,j < r} \theta^{\frac{i(i+r)}{2} + \frac{j(j+r)}{2}} = \sum_{0 \leq i,j < r} \theta^{\frac{i^2+j^2 + (i+j)r}{2}}\end{equation} Here the value of each summand is determined by the value of $i^2 + j^2$ modulo $r$. We now make use of classical formulae for the number of solutions to $i^2+j^2 \equiv a \mod{r}$, which go back at least to the work of Jordan \cite[\S 199 -- 200, pp. 159 -- 161]{jordantraite}, see for example \cite[Lemma 6.24]{LidlNiederreiter}.

For the equation $i^2 + j^2 \equiv 0 \mod{r}$, the number of solutions $(i,j)$ with $0 \leq i,j < r$ is $\alpha = r + (-1)^{(r-1)/2}(r-1)$. Similarly for fixed $a \not\equiv 0 \mod{r}$, there are $\beta = r + (-1)^{(r+1)/2}$ solutions to $i^2+j^2 \equiv a \mod{r}$. Thus~\eqref{eq:doublesum} evaluates to $$\alpha + \beta(\theta + \theta^2 + \cdots + \theta^{r-1}) = \alpha - \beta.$$ Since $\alpha - \beta = (-1)^{(r-1)/2} r$, the lemma follows.
\end{proof}

\begin{lemma}\label{lemma:detD}
Suppose that $\ell \geq 2$. Then $\det(D_{st}) = 1$ for all $1 \leq s < t \leq \ell$.
\end{lemma}

\begin{proof}
From the definition of $D_{st}$ in~\eqref{eq:ABCDE}, one computes directly that $$\det(D_{st}) = \theta^{r^{\ell-2} \sum_{0 \leq \xi_s,\xi_t < r} \xi_s \xi_t} = \theta^{r^{\ell-2} (r(r-1)/2)^2} = 1,$$ as required.
\end{proof}

\section{Basic properties of \texorpdfstring{$G$}{G}}

Let $G$ be the subgroup generated by $\lambda C_t$, $U_t$ (for $1 \leq t \leq \ell$) and $D_{st}$ for $1 \leq s < t \leq \ell$, as in Theorem \ref{thm:mainthm}.

\begin{lemma}\label{lemma:tinvolution}
Let $\sigma: W \rightarrow W$ be the involution defined by $v_{\xi_1, \ldots, \xi_{\ell}} \mapsto v_{-\xi_1, \ldots, -\xi_{\ell}}$. Then the following statements hold:

	\begin{enumerate}[\normalfont (i)]
		\item $\sigma$ normalizes $R$, and acts on $R/Z(R)$ by inversion.
		\item $C_R(\sigma) = Z(R)$.
		\item $\sigma = (-1)^{\ell(r-1)/2} (\lambda C_1)^2 \cdots (\lambda C_{\ell})^2$.
		\item $G$ centralizes $\sigma$.
	\end{enumerate}
	
\end{lemma}

\begin{proof}
A direct calculation using the definition~\eqref{eq:ABCDE} shows that $\sigma A_i \sigma^{-1} = A_i^{-1}$ and $\sigma B_i \sigma^{-1} = B_i^{-1}$ for all $1 \leq i \leq \ell$, from which claim (i) follows.

For (ii), first note that $Z(R) \leq C_R(\sigma)$ since $Z(R)$ consists of scalar matrices. For the other direction, suppose that $x \in C_R(\sigma)$. Then $x = \sigma x \sigma^{-1} = x^{-1}y$ for some $y \in Z(R)$ by (i), so $x^2 \in Z(R)$. It follows then from $r > 2$ that $x \in Z(R)$. Thus we conclude that $C_R(\sigma) = Z(R)$.

Claim (iii) follows from Lemma \ref{lemma:detC2}. For (iv), we check that all the generators of $G$ are centralized by $\sigma$. First since the $C_t$'s commute with each other, it follows from (iii) that $\sigma$ centralizes $\lambda C_t$ for all $1 \leq t \leq \ell$. A direct calculation from the definition of $U_t$ shows that it is centralized by $\sigma$, using the fact that $\xi(\xi+r)/2 \equiv (-\xi)(-\xi+r)/2 \mod{r}$ for all $0 \leq \xi < r$. Similarly $D_{st}$ is centralized by $\sigma$, since $\xi \xi' = (-\xi)(-\xi')$ for all $0 \leq \xi,\xi' < r$.
\end{proof}

\begin{lemma}\label{lemma:Gprops}
Let $Z \leq \GL(W)$ be the group of scalar matrices. The following properties hold for $G$.
	\begin{enumerate}[\normalfont (i)]
		\item $G \cap RZ = G \cap Z$.
		\item $G/(G \cap Z) \cong \Sp_{2\ell}(r)$.
		\item $\det(g)^r = 1$ for all $g \in G$.
		\item $G/[G,G]$ is an $r$-group.
	\end{enumerate}
\end{lemma}

\begin{proof}
Claim (i) follows from Lemma \ref{lemma:tinvolution} (ii) and (iv). Then claim (ii) is a consequence of (i) and the fact that $G$ maps surjectively to $\Sp_{2\ell}(r)$.

For claim (iii), it suffices to check the claim for the generators of $G$. We have $\det(\lambda C_t) = 1$ and $\det(U_t)^r = 1$ by Lemma \ref{lemma:detlemma}, and $\det(D_{st}) = 1$ by Lemma \ref{lemma:detD}. Thus $\det(g)^r = 1$ for all $g \in G$.

It remains to consider (iv). First note that by claim (iii), the subgroup $G \cap Z$ of scalar matrices is an $r$-group. Thus it suffices to check that the abelianization of $G/(G \cap Z)$ is an $r$-group. This follows from (ii), since $\Sp_{2\ell}(r)$ is perfect for $(r,\ell) \neq (3,1)$, and $\Sp_2(3)$ has abelianization of order $r = 3$.\end{proof}

\section{Proof of Theorem \ref{thm:mainthm}}\label{section:mainthm}

In this section, we will prove Theorem \ref{thm:mainthm}. At this point we already know (Lemma \ref{lemma:Gprops}) that $G$ as in Theorem \ref{thm:mainthm} is a central extension of $\Sp_{2\ell}(r)$. For $(r,\ell) \neq (3,1)$ our strategy will be to prove that $G$ is perfect, so in fact $G \cong \Sp_{2\ell}(r)$ since the Schur multiplier of $\Sp_{2\ell}(r)$ is trivial \cite[Theorem 1.1]{Steinberg1981}. In the special case $(r, \ell) = (3,1)$, we will verify $G \cong \SL_2(3)$ by checking that the generators satisfy suitable relations.

We will first prove two lemmas that verify some relations among the generators.

\begin{lemma}\label{lemma:relations}
The following hold for generators of $G$:
	\begin{enumerate}[\normalfont (i)]
		\item $(\lambda C_t)^4 = 1$ for all $1 \leq t \leq \ell$.
		\item $(\lambda C_t U_t)^3 = 1$ for all $1 \leq t \leq \ell$.
	\end{enumerate}
\end{lemma}

\begin{proof}
Claim (i) follows from Lemma \ref{lemma:detC2}. For claim (ii), from the definition of $C_t$ and $U_t$ as tensor products, it suffices to prove the lemma for $\ell = 1$. Suppose then that $\ell = 1$, in which case we aim to prove that $(\lambda C U)^3 = 1$. Under the homomorphism $\pi: G \rightarrow \Sp_{2}(r)$ from~\eqref{eq:pimap}, we have $$\pi(C) = \begin{pmatrix} 0 & 1 \\ -1 & 0 \end{pmatrix},\ \pi(U) = \begin{pmatrix} 1 & 1 \\ 0 & 1 \end{pmatrix}.$$ Thus $\pi((CU)^3)$ is trivial, so $(CU)^3$ is a scalar matrix by Lemma \ref{lemma:Gprops} (i) and~\eqref{eq:kerpi}. To check what the scalar multiplier is, we compute the $(0,0)$ entry of $(CU)^3$ with respect to the basis $v_0$, $v_1$, $\ldots$, $v_{r-1}$. To this end, we have $((CU)^3)_{00}$ equal to \begin{align*}\sum_{0 \leq i,j < r} (CU)_{0i}(CU)_{ij}(CU)_{j0} &= \sum_{0 \leq i,j < r} \theta^{\frac{i(i+r)}{2} + ij + \frac{j(j+r)}{2}} \\ &= \sum_{0 \leq i,j < r} \theta^{\frac{(i+j)(i+j+r)}{2}} \\ &= r \sum_{0 \leq i < r} \theta^{\frac{i(i+r)}{2}} = r \operatorname{Tr}(U),\end{align*} so $(CU)^3 = r \operatorname{Tr}(U) I_r$.

Now since $(\lambda CU)^3$ is a scalar matrix, by Lemma \ref{lemma:Gprops} (iii) it has $r$-power order. So since $r > 2$, it suffices to check that $(\lambda CU)^6 = 1$. For this, we first get \begin{equation}\label{eq:CE6}(\lambda CU)^6 = \lambda^6 (CU)^6 = \lambda^6 r^2 \operatorname{Tr}(U)^2 I_r = \lambda^6 r^3 (-1)^{(r-1)/2} I_r\end{equation} with Lemma \ref{lemma:tracesquare}. On the other hand, we have $$\lambda^6 = \left(r^{r-1} \det(C)^{-2}\right)^3 = \left(r^{r-1} r^{-r} (-1)^{(r-1)/2} \right)^3 = r^{-3} (-1)^{(r-1)/2}$$ by Lemma \ref{lemma:detC2}. Combining this with~\eqref{eq:CE6}, we get $(\lambda CU)^6 = 1$.\end{proof}

\begin{remark}\label{remark:detformula}
As seen in the proof of Lemma \ref{lemma:relations}, we have $(CU)^3 = r \operatorname{Tr}(U) I_r$. On the other hand $(\lambda CU)^3 = 1$ by Lemma \ref{lemma:relations}, so $\lambda^3 r \operatorname{Tr}(U) = 1$. Since $\lambda^2 = (-1)^{(r-1)/2} r^{-1}$ by Lemma \ref{lemma:detC2}, it follows that $\lambda (-1)^{(r-1)/2} \operatorname{Tr}(U) = 1$. Then plugging in the definition~\eqref{eq:lambda} of $\lambda$ gives the formula $$\det(C) = r^{(r-1)/2} \sum_{0 \leq i < r} \theta^{\frac{i(i+r)}{2}} = \legendre{2}{r} r^{(r-1)/2} \sum_{0 \leq i < r} \theta^{i^2},$$ where the second equality follows from the fact that $\sum_{0 \leq i < r} \theta^i = 0$.
\end{remark}

\begin{lemma}\label{lemma:relations2}
Suppose that $\ell \geq 2$. Then the following hold the for generators of $G$:
	\begin{enumerate}[\normalfont (i)]
		\item $((\lambda C_s)^2 D_{st})^2 = 1$ for all $1 \leq s < t \leq \ell$.
		\item Denote $X = C_sD_{st}C_s^{-1}$,where $1 \leq s < t \leq \ell$. Then $XU_sX^{-1}U_s^{-1} = U_t D_{st}$.
		\item Denote $Y = C_tD_{st}C_t^{-1}$,where $1 \leq s < t \leq \ell$. Then $YU_tY^{-1}U_t^{-1} = U_s D_{st}$.
	\end{enumerate}
\end{lemma}

\begin{proof}
From the definition of $C_s$ and $D_{st}$ acting on a tensor product, it is clear that it suffices to prove the lemma in the case $\ell = 2$. Then $(s,t) = (1,2)$, and it follows from Lemma \ref{lemma:detC2} that the action of $(\lambda C_1)^2 D_{12}$ on $W$ is given by $$v_{\xi_1, \xi_2} \mapsto (-1)^{(r-1)/2} \theta^{\xi_1 \xi_2} v_{-\xi_1,\xi_2}$$ for all $0 \leq \xi_1,\xi_2 < r$. This clearly defines an involution, so (i) holds.

For claim (ii), we first note that a direct calculation from the definitions~\eqref{eq:ABCDE} shows that $X = C_1 D_{12} C_1^{-1}$ acts on $W$ by $$v_{\xi_1,\xi_2} \mapsto v_{\xi_1-\xi_2, \xi_2}$$ for all $0 \leq \xi_1, \xi_2 < r$. Note that then $X^{-1}$ acts on $W$ by $v_{\xi_1,\xi_2} \mapsto v_{\xi_1+\xi_2, \xi_2}$. Then a straightforward calculation shows that $XU_1X^{-1}U_1^{-1}$ acts on $W$ by $$v_{\xi_1,\xi_2} \mapsto \theta^{\frac{\xi_2(\xi_2+r)}{2} + \xi_1\xi_2} v_{\xi_1,\xi_2}$$ for all $0 \leq \xi_1, \xi_2 < r$. In other words, we have $XU_1X^{-1}U_1^{-1} = U_2 D_{12}$, as required.

Claim (iii) follows essentially with the same calculations as (ii), noting that $Y = C_2 D_{12} C_2^{-1}$ acts on $W$ by $v_{\xi_1,\xi_2} \mapsto v_{\xi_1, \xi_2-\xi_1}$ for all $0 \leq \xi_1, \xi_2 < r$.
\end{proof}

With this we can prove the main result.

\begin{proof}[Proof of Theorem \ref{thm:mainthm}]
We already know by Theorem \ref{thm:jordanthm} that the restriction of $\pi$ to $G$ defines a surjective map $G \rightarrow \Sp_{2\ell}(r)$. Thus to prove the theorem, it suffices to prove that $G \cong \Sp_{2\ell}(r)$.

We first consider the case where $r = 3$ and $\ell = 1$, so $G = \langle \lambda C_1, U_1 \rangle$. We will make use of the fact that $\Sp_2(3) = \SL_2(3)$ has a presentation \begin{equation}\label{eq:sl2_3presentation}\Sp_2(3) \cong \langle x,y\ |\ x^4 = y^3 = (xy)^3 = [x^2,y] = 1 \rangle.\end{equation} Now note that $(\lambda C_1)^4 = (\lambda C_1 U_1)^3 = 1$ by Lemma \ref{lemma:relations}, and $U_1^3 = 1$ is clear from the action of $U_1$. It follows from Lemma \ref{lemma:tinvolution} that $(\lambda C_1)^2$ is central in $G$, so by~\eqref{eq:sl2_3presentation} there is a surjective homomorphism $\Sp_2(3) \rightarrow G$. By Lemma \ref{lemma:Gprops} (iii) this is an isomorphism, which proves the theorem in this case.

Suppose then that $(r, \ell) \neq (3,1)$, in which case $\Sp_{2\ell}(r)$ is perfect. Moreover, we know from Lemma \ref{lemma:Gprops} (ii) that $G$ is a central extension of $\Sp_{2\ell}(r)$. Thus it will suffice to verify that $G$ is perfect, since the Schur multiplier of $\Sp_{2\ell}(r)$ is trivial \cite[Theorem 1.1]{Steinberg1981}, and thus $\Sp_{2\ell}(r)$ has no non-trivial perfect central extension. 


For the proof that $G$ is perfect, suppose first that $\ell > 1$. Now $G/[G,G]$ is an $r$-group by Lemma \ref{lemma:Gprops} (iv), and on the other hand $\lambda C_t$ has order coprime to $r$ by Lemma \ref{lemma:relations}, so $\lambda C_t \in [G,G]$ for all $1 \leq t \leq \ell$. Similarly $(\lambda C_s)^2 D_{st}$ has order coprime to $r$ by Lemma \ref{lemma:relations2} (i), so $(\lambda C_s)^2 D_{st} \in [G,G]$ and thus $D_{st} \in [G,G]$ for all $1 \leq s < t \leq \ell$. By Lemma \ref{lemma:relations2} (ii) we have $U_t D_{st} \in [G,G]$ for all $1 \leq s < t \leq \ell$, so we conclude that $U_t \in [G,G]$ for all $1 < t \leq \ell$. Finally by Lemma \ref{lemma:relations2} (iii) we have $U_1 D_{12} \in [G,G]$, so $U_1 \in [G,G]$. Since $[G,G]$ contains all the generators of $G$, we have $G = [G,G]$.

It remains to consider the case $\ell = 1$, so here $r > 3$ by assumption. We have $G = \langle \lambda C_1, U_1 \rangle$ and it follows from Lemma \ref{lemma:relations} that $\lambda C_1$ and $\lambda C_1 U_1$ both have order coprime to $r$. Since $G/[G,G]$ is an $r$-group (Lemma \ref{lemma:Gprops} (iv)), we conclude as in the previous paragraph that $G = [G,G]$. This completes the proof of the theorem.
\end{proof}

\section{Irreducible submodules of degrees \texorpdfstring{$(r^\ell \pm 1)/2$}{(rl ± 1)/2}}\label{section:submodules}

Let $G \cong \Sp_{2\ell}(r)$ be as in Theorem \ref{thm:mainthm}. Then the $\F[G]$-submodule structure of $W$ is known to be as follows, see for example \cite[Theorem 2.5]{Ward1972} and \cite[Lemma 5.2]{GuralnickMagaardSaxlTiep2002}.

	\begin{itemize}
		\item If $\operatorname{char} \F \neq 2$, then $W$ is a direct sum of two irreducible $\F[G]$-submodules of dimensions $(r^{\ell}-1)/2$ and $(r^{\ell}+1)/2$.
		\item If $\operatorname{char} \F = 2$, then $W$ is uniserial, with composition factors of dimensions $(r^{\ell}-1)/2$, $1$, $(r^{\ell}-1)/2$. Specifically $$0 \subsetneqq A \subsetneqq B \subsetneqq W$$ with $A \cong W/B$ irreducible of dimension $(r^{\ell}-1)/2$, and $\dim B/A = 1$. 
	\end{itemize}

The irreducible modules of dimensions $(r^{\ell} \pm 1)/2$ are known as \emph{irreducible Weil representations} of $\Sp_{2\ell}(r)$. The submodules of $W$ can be described as follows.

Let $\sigma$ be the involution $v_{\xi_1,\ldots,\xi_{\ell}} \mapsto v_{-\xi_1,\ldots,-\xi_{\ell}}$ defined in Lemma \ref{lemma:tinvolution}. Since $G$ centralizes $\sigma$ (Lemma \ref{lemma:tinvolution} (iv)), it follows that $G$ acts on the eigenspaces of $\sigma$. We will denote the $\pm 1$-eigenspace of $\sigma$ by $W^{\pm}$. 

Suppose that $\operatorname{char} \F \neq 2$. Then it is straightforward to see that \begin{align*} W^+ &= \langle v_{\xi_1, \ldots, \xi_\ell} + v_{-\xi_1, \ldots, -\xi_{\ell}} : 0 \leq \xi_1, \ldots, \xi_{\ell} < r \rangle, \\ W^- &= \langle v_{\xi_1, \ldots, \xi_\ell} - v_{-\xi_1, \ldots, -\xi_{\ell}} : 0 \leq \xi_1, \ldots, \xi_{\ell} < r \rangle.\end{align*} Furthermore, we have $W = W^+ \oplus W^{-}$. Here $W^{\pm}$ is an irreducible $\F[G]$-submodule of dimension $(r^{\ell} \pm 1)/2$.

Suppose then that $\operatorname{char} \F = 2$. In this case the non-zero proper submodules of $W$ are \begin{align*} A &= \langle v_{\xi_1, \ldots, \xi_\ell} + v_{-\xi_1, \ldots, -\xi_{\ell}} : 0 \leq \xi_1, \ldots, \xi_{\ell} < r \rangle, \\ B &= W^+ = \langle A, v_{0,\ldots,0} \rangle. \end{align*} Here $A$ is an irreducible $\F[G]$-module of dimension $(r^{\ell}-1)/2$. Furthermore $B/A$ is $1$-dimensional, and $W/B \cong A$ is irreducible of dimension $(r^{\ell}-1)/2$.

On each submodule, we can easily write down the action of the generators of $G$ on a basis, thus giving generating matrices for the irreducible Weil representation. For example if $\operatorname{char} \F \neq 2$, then $W^+$ has a basis consisting of the vectors $$y_{\xi_1, \ldots, \xi_{\ell}} := v_{\xi_1,\ldots,\xi_{\ell}} + v_{-\xi_1,\ldots,-\xi_{\ell}},$$ where $0 \leq \xi_1,\ldots,\xi_{\ell} < r$. It is then easy to see from~\eqref{eq:ABCDE} and~\eqref{eq:U} that similarly to the action of $G$ on the basis vectors of $W$, we have \begin{align*} C_t y_{\xi_1, \ldots, \xi_{\ell}} &= \sum_{0 \leq i < r} \theta^{i \xi_t} y_{\xi_1, \ldots, \xi_{t-1}, i, \xi_{t+1}, \ldots, \xi_{\ell}} \\ D_{st} y_{\xi_1, \ldots, \xi_{\ell}} &= \theta^{\xi_s \xi_t} y_{\xi_1, \ldots, \xi_{\ell}} \\ U_t v_{\xi_1, \ldots, \xi_{\ell}} &= \theta^{\frac{\xi_t(\xi_t+r)}{2}} y_{\xi_1, \ldots, \xi_{\ell}}  \end{align*} for all $0 \leq \xi_1,\ldots,\xi_{\ell} < r$.


\begin{thebibliography}{10}

\bibitem{BoltRoomWall1961}
B.~Bolt, T.~G. Room, and G.~E. Wall.
\newblock On the {C}lifford collineation, transform and similarity groups. {I},
  {II}.
\newblock {\em J. Austral. Math. Soc.}, 2:60--79, 80--96, 1961/62.

\bibitem{MAGMA}
W.~Bosma, J.~Cannon, and C.~Playoust.
\newblock The {M}agma algebra system. {I}. {T}he user language.
\newblock {\em J. Symbolic Comput.}, 24(3-4):235--265, 1997.

\bibitem{DoerkHawkes}
K.~Doerk and T.~Hawkes.
\newblock {\em Finite soluble groups}, volume~4 of {\em De Gruyter Expositions
  in Mathematics}.
\newblock Walter de Gruyter \& Co., Berlin, 1992.

\bibitem{Gerardin1977}
P.~G\'{e}rardin.
\newblock Weil representations associated to finite fields.
\newblock {\em J. Algebra}, 46(1):54--101, 1977.

\bibitem{GlasbyHowlett1992}
S.~P. Glasby and R.~B. Howlett.
\newblock Extraspecial towers and {W}eil representations.
\newblock {\em J. Algebra}, 151(1):236--260, 1992.

\bibitem{GuralnickMagaardSaxlTiep2002}
R.~M. Guralnick, K.~Magaard, J.~Saxl, and P.~H. Tiep.
\newblock Cross characteristic representations of symplectic and unitary
  groups.
\newblock {\em J. Algebra}, 257(2):291--347, 2002.

\bibitem{HoltRoneyDougal2005}
D.~F. Holt and C.~M. Roney-Dougal.
\newblock Constructing maximal subgroups of classical groups.
\newblock {\em LMS J. Comput. Math.}, 8:46--79, 2005.

\bibitem{Howe1973}
R.~E. Howe.
\newblock On the character of {W}eil's representation.
\newblock {\em Trans. Amer. Math. Soc.}, 177:287--298, 1973.

\bibitem{Isaacs1973}
I.~M. Isaacs.
\newblock Characters of solvable and symplectic groups.
\newblock {\em Amer. J. Math.}, 95:594--635, 1973.

\bibitem{jordantraite}
C.~Jordan.
\newblock {\em Trait\'{e} des substitutions et des \'{e}quations
  alg\'{e}briques}.
\newblock Gauthier-Villars, Paris, 1870.

\bibitem{jordan1908recherches}
C.~Jordan.
\newblock Recherches sur les groupes r{\'e}solubles.
\newblock {\em Memorie della Pontificia Accademia Romana dei Nuovi Lincei},
  26:7--39, 1908.

\bibitem{KleidmanLiebeck}
P.~Kleidman and M.~Liebeck.
\newblock {\em The subgroup structure of the finite classical groups}, volume
  129 of {\em London Mathematical Society Lecture Note Series}.
\newblock Cambridge University Press, Cambridge, 1990.

\bibitem{KorhonenBook}
M.~Korhonen.
\newblock {\em Maximal solvable subgroups of finite classical groups}, volume
  2346 of {\em Lecture Notes in Mathematics}.
\newblock Springer, Cham, 2024.

\bibitem{KorhonenGithub}
M.~Korhonen.
\newblock Weil{R}epresentation, {G}it{H}ub repository, 2025.
\newblock \url{https://github.com/korhonenmikko/WeilRepresentation/}.

\bibitem{Landau1958}
E.~Landau.
\newblock {\em Elementary number theory}.
\newblock Chelsea Publishing Co., New York, 1958.
\newblock Translated by J. E. Goodman.

\bibitem{LidlNiederreiter}
R.~Lidl and H.~Niederreiter.
\newblock {\em Finite fields}, volume~20 of {\em Encyclopedia of Mathematics
  and its Applications}.
\newblock Cambridge University Press, Cambridge, second edition, 1997.
\newblock With a foreword by P. M. Cohn.

\bibitem{Rose1994}
H.~E. Rose.
\newblock {\em A course in number theory}.
\newblock Oxford Science Publications. The Clarendon Press, Oxford University
  Press, New York, second edition, 1994.
	
\bibitem{Schur1921}
I.~Schur.
\newblock {\"U}ber die {Gau{{\ss}}schen} {Summen}.
\newblock {\em Nachr. Ges. Wiss. G{\"o}ttingen, Math.-Phys. Kl.},
  1921:147--153, 1921.

\bibitem{Steinberg1981}
R.~Steinberg.
\newblock Generators, relations and coverings of algebraic groups. {II}.
\newblock {\em J. Algebra}, 71(2):527--543, 1981.

\bibitem{SuprunenkoBook}
D.~A. Suprunenko.
\newblock {\em Matrix groups}.
\newblock American Mathematical Society, Providence, R.I., 1976.
\newblock Translated from the Russian, Translation edited by K. A. Hirsch,
  Translations of Mathematical Monographs, Vol. 45.

\bibitem{Szechtman1998}
F.~Szechtman.
\newblock Weil representations of the symplectic group.
\newblock {\em J. Algebra}, 208(2):662--686, 1998.

\bibitem{Ward1972}
H.~N. Ward.
\newblock Representations of symplectic groups.
\newblock {\em J. Algebra}, 20:182--195, 1972.

\bibitem{Weil1964}
A.~Weil.
\newblock Sur certains groupes d'op\'{e}rateurs unitaires.
\newblock {\em Acta Math.}, 111:143--211, 1964.

\end{thebibliography}
\end{document}